\documentclass[leqno,a4paper]{article}

\usepackage{layout}
\usepackage[OT1]{fontenc}
\usepackage{color}
\usepackage{amsthm,amsmath,graphicx,latexsym,amssymb,amscd,amsfonts,enumerate, bm}
\usepackage[colorlinks,citecolor=blue,linkcolor=red,urlcolor=blue]{hyperref}

\theoremstyle{plain}
\newtheorem{theorem}{Theorem}
\newtheorem{proposition}[theorem]{Proposition}
\newtheorem{lemma}[theorem]{Lemma}

\newtheorem{remark}[theorem]{Remark}
\newtheorem{example}[theorem]{Example}

\makeatletter
\def\section{\@startsection {section}{1}{\z@}{-3.5ex plus -1ex minus 
-.2ex}{2.3ex plus .2ex}{\large\bf}}
\makeatother

\makeatletter
\def\subsection{\@startsection {subsection}{1}{\z@}{-3.5ex plus -1ex minus 
-.2ex}{2.3ex plus .2ex}{\normalsize\bf}}
\makeatother

\begin{document}

\title{\textbf{\large Small ball probabilities for a  class of\\ time-changed self-similar processes}}
\author{\textsc{Kei Kobayashi}\thanks{Department of Mathematics, University of Tennessee, 1403 Circle Drive, Knoxville, TN 37996, USA. \textit{Email:} kkobayas@utk.edu}\\
        \vspace{1mm} \\ \textit{\normalsize University of Tennessee}}
%%\date{ }
\maketitle                                                       
\vspace{-2mm}
\renewcommand{\thefootnote}{\fnsymbol{footnote}}

\begin{abstract}
This paper establishes small ball probabilities for a  class of  time-changed processes $X\circ E$, where $X$ is a self-similar process and $E$ is an independent continuous process, each with a certain small ball probability.  In particular, examples of the outer process $X$ and the time change $E$ include an iterated fractional Brownian motion and the inverse of a general subordinator with infinite L\'evy measure, respectively. The small ball probabilities of such time-changed processes show power law decay, and the rate of decay does not depend on the small deviation order of the outer process $X$, but on the self-similarity index of $X$.      \footnote[0]{\textit{AMS 2010 subject classifications:} Primary 60F99, 60G18; secondary 60G51. 
     \textit{Keywords:} small ball probability, small deviation, time-changed self-similar process, iterated process, inverse subordinator.}
\end{abstract}

\large

\section{Introduction}

Let $W$ be a one-dimensional standard Brownian motion and let $E_\beta$ be the inverse of a stable subordinator $D_\beta$ of index $\beta\in(0,1)$, independent of $W$.
% (precise definition to be provided in Example \ref{Example_stable}). 
%Using the idea of how to handle small deviations of a class of iterated processes presented by Aurzada and Lifshits \cite{AurzadaLifshits}, 
Nane \cite{Nane_LIL2009} established that the small ball probability of the time-changed Brownian motion $W\circ E_\beta$ is given by
\begin{align}\label{Nane}
	 \mathbb{P}\Bigl(\sup_{0\le t\le 1} |W(E_\beta(t))|\le \epsilon \Bigr)
	\sim \dfrac{32 \Gamma(\beta)\sin(\beta\pi)}{\pi^4} \sum_{k=1}^\infty \dfrac{(-1)^{k-1}}{(2k-1)^3}\epsilon^2 \ \ \textrm{as} \ \ \epsilon\downarrow 0, 
\end{align}
where $\Gamma(\cdot)$ is Euler's Gamma function and the notation $f(x)\sim g(x)$ for two positive functions $f$ and $g$ means that $\lim f(x)/g(x)=1$. 
%
%$f(x)\approx g(x)$ means that \[
%	0<\liminf f(x)/g(x)\le \limsup f(x)/g(x)<\infty.
%\]
%
%
%\begin{align}\label{Nane1}
%	\mathbb{P}\Bigl(\sup_{0\le t\le 1} |W_H(E_\beta(t))|\le \epsilon \Bigr)
%	\approx \epsilon^{1/H}  \ \ \textrm{as} \ \ \epsilon\downarrow 0,
%\end{align}
The result is interesting since the small ball probability of $W\circ E_\beta$ shows power law decay unlike the exponential decay observed for the original Brownian motion $W$:
\begin{align}\label{Nane2}
	-\log \mathbb{P}\Bigl(\sup_{0\le t\le 1} |W(t)|\le \epsilon \Bigr)
	\sim \dfrac{\pi^2}{8} \epsilon^{-2}  \ \ \textrm{as} \ \ \epsilon\downarrow 0.
\end{align}
Moreover, the rate of decay in \eqref{Nane} does not depend on the stability index $\beta$ of the underlying stable subordinator $D_\beta$; the dependence on $\beta$ only appears as a small deviation constant independent of $\epsilon$.

% and $c_0$ is a positive constant appearing in the following expression concerning the small ball probability of the Brownian motion $W$:

The proof of \eqref{Nane} provided in \cite{Nane_LIL2009} essentially relies on %the well-known fact that $E_\beta$ has the Mittag--Leffler distribution and 
the following expression for the Laplace transform of the random variable $E_\beta(1)$ and its asymptotic behavior along the negative real axis (see e.g.\ Proposition 1(a) of \cite{Bingham} and Theorem 1.4 of \cite{Podlubny}):
\begin{align}\label{Mittag-Leffler_asymp}
	\mathbb{E}[e^{-aE_\beta(1)}]=\mathbf{E}_\beta(-a)\sim \dfrac{1}{a\Gamma(1-\beta)} \ \ \textrm{as} \ \ a\to\infty.
\end{align}
Here, $\mathbf{E}_\beta(z)=\sum_{n=0}^\infty z^n/\Gamma(n\beta+1)$ is the Mittag--Leffler function with parameter $\beta$.
%which implies that 
%\[
%	\rho=\lim_{a\to\infty} a\mathbb{E}[e^{-aE(1)}]=\dfrac{1}{\Gamma(1-\beta)}.
%\]
%The last expression coincides with $\Gamma(\beta)\sin(\beta\pi)/\pi$ due to Euler's reflection formula; consequently, the expression \eqref{small_ball_2} takes the specific form given in \eqref{Nane}. 
Nane \cite{Nane_LIL2009} also extended the result to a time-changed process $X\circ E_\beta$, where the outer process $X$ is a self-similar process possessing a certain small ball probability, which particularly includes the case of a fractional Brownian motion. 
However, the exact small deviation constant cannot be specified unlike the situations considered in \cite{AurzadaLifshits}; see Remark \ref{Remark_main4} for details of this point.

The main motivation to analyze such time-changed processes comes from their non-standard diffusion structures. In particular, the time-changed Brownian motion $W\circ E_\beta$  is non-Gaussian and non-Markovian, and is widely used to model subdiffusions, where particles spread more slowly than the classical Brownian particles do. Namely, the particles represented by the time-changed Brownian motion are trapped and immobile during the constant periods of the time change $E_\beta$. One interesting aspect of the time-changed Brownian motion is that its transition probabilities satisfy the following time-fractional generalization of the Fokker--Planck or forward Kolmogorov equation:
\[
	\partial^\beta_t p(t,x)= \frac 12 \partial_x^2 p(t,x), \ \ t>0, \ x\in\mathbb{R}.
\]
Here, $\partial^\beta_t$ denotes the Caputo fractional derivative operator in time of order $\beta$ (see e.g.\ \cite{Podlubny}). 
The correspondence between the time-changed Brownian motion and the fractional Kolmogorov equation  has been extended to those for different classes of time-changed processes and stochastic differential equations they drive; see e.g.\ \cite{HKRU,HKU-2,HKU-1,Kobayashi,Magdziarz_spa,MS_1,MS_2}. The fractional Kolmogorov equations have found many applications in a wide range of scientific areas, including physics, \cite{MetzlerKlafter00,Zaslavsky}, finance \cite{GMSR,Magdziarz_BS}, hydrology \cite{BWM}, and biology \cite{Saxton}. 
%This fact has been extended to several different classes of time-changed processes and stochastic differential equations they drive; see e.g.\ \cite{HKRU,HKU-2,HKU-1,Kobayashi,Magdziarz_spa,MS_1,MS_2}. These fractional partial differential equations have found many applications in a wide range of scientific areas, including physics, \cite{MetzlerKlafter00,Zaslavsky}, finance \cite{GMSR,Magdziarz_BS}, hydrology \cite{BWM}, and biology \cite{Saxton}. 

In this paper, we establish small ball probabilities for a class of time-changed processes $X\circ E$, where $X$ is a self-similar process and $E$ is a continuous process independent of $X$, each with a certain small ball probability (Theorems \ref{Theorem_main2} and \ref{Theorem_main4}). This largely extends the results in \cite{Nane_LIL2009} in terms of both the outer process $X$ and the time change $E$. Examples of $X$ and $E$ that can be handled within our framework include an iterated fractional Brownian motion and the inverse of a general subordinator with infinite L\'evy measure, % having no atoms, 
respectively. Our strategy is to employ a version of the Tauberian theorem (Lemma \ref{Lemma_Tauberian_2}) along with a general fact concerning a subordinator (Proposition \ref{Proposition_E}), which is a different approach from what was taken in \cite{Nane_LIL2009} to derive \eqref{Nane}. In particular, even when $E$ is the inverse of a stable subordinator, our method does not rely on the asymptotic expression for the Mittag--Leffler function given in \eqref{Mittag-Leffler_asymp}.

The results to be established in this paper show that the small ball probability of a certain time-changed process $X\circ E$ has power law decay whose rate depends on the self-similarity index of the outer process $X$, but not on the small deviation order of $X$. In a particular case of a time-changed Brownian motion $W\circ E$ with the time change $E$ being the inverse of a general subordinator with infinite L\'evy measure, the dependence on $E$ is reflected on the associated small deviation constant.  We will specify that constant when the underlying subordinator is a Gamma subordinator or a tempered stable subordinator; these specific time changes have been recently investigated to analyze anomalous diffusions observed in various natural phenomena (see e.g.\ \cite{Janczura}). This will allow us to examine how the small ball probabilities for the important subclasses of time-changed processes vary according to the choice of the parameters defining the underlying subordinators.  In particular, our result with the time change being the inverse of a tempered stable subordinator recovers \eqref{Nane} as an immediate corollary; see Remark \ref{Remark_tempered} for details.

\section{Small ball probabilities for time-changed Brownian motions}

Let $E$ be a stochastic process in $\mathbb{R}^1$ with continuous, nondecreasing paths starting at 0. One way to construct such a process is through a subordinator.  Namely, let $D$ be a subordinator with Laplace exponent $\psi$ and infinite L\'evy measure $\nu$; i.e.\ $D$ is a one-dimensional nondecreasing L\'evy process with c\`adl\`ag paths starting at 0 with Laplace transform 
\begin{align}\label{def_LaplaceExponent}
	\mathbb{E}[e^{-sD(t)}]=e^{-t\psi(s)}, \ \ \textrm{where} \ \ 
	\psi(s)=b s +\int_0^\infty (1-e^{-sx})\hspace{1pt}\nu(\textrm{d}x), \ \ s> 0,
\end{align}
with $b\ge 0$ and
$\int_0^\infty (x\wedge 1)\hspace{1pt} \nu(\textrm{d}x)<\infty$. 
The assumption that the L\'evy measure is infinite (i.e.\ $\nu(0,\infty)=\infty$) implies that $\psi$ is an increasing function with $\lim_{s\to\infty}\psi(s)=\infty$ and $D$ has strictly increasing paths with infinitely many jumps (see e.g.\ Theorem 21.3 of \cite{Sato}). Let $E$ be the \textit{inverse} or \textit{first hitting time process} of $D$; i.e.\ 
\[
	E(t):=\inf\{u>0; D(u)>t\}, \ \ t\ge 0.
\]
Since $D$ has strictly increasing paths, the process $E$, called an \textit{inverse subordinator}, has continuous, nondecreasing paths starting at 0 (see e.g.\ Lemma 2.7 of \cite{Kobayashi}). It is known that $E$ generally does not have independent or stationary increments (see Section 3 of \cite{MS_1}), which implies that even if $X$ is a Gaussian or L\'evy process independent of $E$, the time-changed process $X\circ E$ no longer has the same structure. Hence, existing results on small ball probabilities of Gaussian or L\'evy processes cannot be directly applied to find the small ball probability of $X\circ E$.

A stochastic process $X$ in $\mathbb{R}^1$ is called a \textit{self-similar process of index $H>0$} if for every $a>0$, $(X(at))_{t\ge 0}=^d (a^H X(t))_{t\ge 0}$. Important examples of self-similar processes include fractional Brownian motions, iterated fractional Brownian motions, and stable L\'evy processes. Brief definitions of these processes will be provided in examples in Section \ref{Section_Extensions}.  

%Throughout the paper, we use the following notations concerning asymptotic ratios of two positive functions $f$ and $g$: 
%\begin{quote}
%\begin{itemize}
%\item $f(x)\sim g(x)$ means that $\lim f(x)/g(x)=1;$
%\item $f(x)\approx g(x)$ means that $0<\liminf f(x)/g(x)\le \limsup f(x)/g(x)<\infty$.
%\end{itemize}
%\end{quote}

The following theorem largely extends Theorem 2.1 of \cite{Nane_LIL2009} to the case when the time change is given by the inverse of a non-stable subordinator.

\begin{theorem}\label{Theorem_main2}
 Let $E$ be the inverse of a subordinator $D$ with infinite L\'evy measure $\nu$, independent of a one-dimensional standard Brownian motion $W$.  Then for all $T>0$ at which $\nu$ has no mass (i.e.\ $\nu(\{T\})=0)$,  
\begin{align}\label{small_ball_2}
	 \mathbb{P}\Bigl(\sup_{0\le t\le T} |W(E(t))|\le \epsilon \Bigr)
	\sim\dfrac{32}{\pi^3} \nu(T,\infty) \sum_{k=1}^\infty \dfrac{(-1)^{k-1}}{(2k-1)^3}\epsilon^2
	 \ \ \textrm{as} \ \ \epsilon\downarrow 0. 
\end{align}
This is interpreted as $\mathbb{P}(\sup_{0\le t\le T} |W(E(t))|\le \epsilon)=o(\epsilon^2)$ if $\nu(T,\infty)=0$.
\end{theorem}

\begin{remark}
\begin{em}
1) If $\nu(T,\infty)>0$, then the small ball probability of the time-changed Brownian motion $W\circ E$ has a power law decay. 
Moreover, the rate of decay of the small ball probability does not depend on the choice of the inverse subordinator $E$; the dependence on $E$ is reflected only on the constant $\nu(T,\infty)$. 

2) In the degenerate case when $E(t)= t$, clearly $\nu\equiv 0$ and hence the small ball probability becomes $o(\epsilon^2)$. This is because the small ball probability for the Brownian motion $W$ (without a time change) has an exponential decay as in \eqref{Nane2}.

3) If $E=E_\beta$ is the inverse of a $\beta$-stable subordinator $D_\beta$, then 
\eqref{small_ball_2} immediately recovers \eqref{Nane}. Indeed, 
using the explicit form of the L\'evy measure of $D_\beta$ (see e.g.\ Example 1.3.18 of \cite{Applebaum}), we observe that 
\begin{align}\label{nu_stable}
  	\nu(T,\infty)= \int_T^\infty \dfrac{\beta}{\Gamma(1-\beta)} x^{-1-\beta}\hspace{1pt} \textrm{d}x 
	=\dfrac{T^{-\beta}}{\Gamma(1-\beta)}.
  \end{align}
When $T=1$, the last expression coincides with $\Gamma(\beta)\sin(\beta\pi)/\pi$ due to Euler's reflection formula; consequently, the expression \eqref{small_ball_2} takes the specific form given in \eqref{Nane}.  
\end{em}
\end{remark}

The proof of Theorem \ref{Theorem_main2} requires some auxiliary facts to be established first.

\begin{lemma}[A version of the Tauberian theorem]\label{Lemma_Tauberian_2}
Let $V$ be a nonnegative random variable and let $A$ and $\theta$ be positive constants. Then 
\[
	\mathbb{E}[e^{-aV}]\sim A\hspace{1pt} a^{-\theta} \ \ \textrm{as} \ \ a\to\infty 
\]
if and only if 
\[
	\mathbb{P}(V\le \epsilon) \sim \frac{A}{\Gamma(\theta+1)}\hspace{1pt}\epsilon^\theta \ \ \textrm{as} \ \ \epsilon\downarrow 0.
\]
\end{lemma}

\begin{proof}
This follows from Corollary 1a and Theorem 4.3 of Chapter V of \cite{Widder}.
\end{proof}

\begin{proposition}\label{Proposition_E}
Let $E$ be the inverse of a subordinator $D$ with infinite L\'evy measure $\nu$. 
Then %$E$ is a continuous, nondecreasing process starting at 0 such that 
for all $T>0$ at which $\nu$ has no mass,
\begin{align}\label{small_E(1)_sub}
	\mathbb{P}(E(T)\le \epsilon) 
	\sim \nu(T,\infty)\hspace{1pt} \epsilon \ \ \textrm{as} \ \ \epsilon\downarrow 0.
\end{align}
This is interpreted as $\mathbb{P}(E(T)\le \epsilon) =o(\epsilon)$ if $\nu(T,\infty)=0$.
%Hence, \eqref{small_E(1)} holds with $\rho=\nu(1,\infty)$.
\end{proposition}

\begin{proof}
See Appendix. 
\end{proof}

\begin{lemma}\label{Lemma_laplace}
Let $E$ be the inverse of a subordinator $D$ with Laplace exponent $\psi$. 
%%%%and infinite L\'evy measure. 
Then for any fixed $a>0$, the Laplace transform of the function $t\mapsto \mathbb{E}[e^{-aE(t)}]$ exists and is given by 
\begin{align}\label{Laplace_E(t)}
	\mathcal{L}_t\bigl[\mathbb{E}[e^{-aE(t)}]\bigr](s)=\frac{\psi(s)}{s}(\psi(s)+a)^{-1}, \ \ s>0. 
\end{align}
\end{lemma}

\begin{proof}
See Appendix. 
\end{proof}

\begin{remark}
\begin{em}
Lemma \ref{Lemma_laplace} implies that if $E=E_\beta$ is the inverse of a $\beta$-stable subordinator, then for a fixed $a>0$, 
\begin{align*}
	\mathcal{L}_t\bigl[\mathbb{E}[e^{-aE_\beta(t)}]\bigr](s)%=\frac{s^\beta}{s}(s^\beta+a)^{-1}
	=\dfrac{s^{\beta-1}}{s^\beta+a}, \ \ s>0. 
\end{align*}
Since the right hand side coincides with the Laplace transform of the function $t\mapsto\mathbf{E}_\beta(-at^\beta)$ (see e.g.\ \cite{Podlubny}), we recover the well-known formula
$
	\mathbb{E}[e^{-aE_\beta(t)}]=\mathbf{E}_\beta(-at^\beta),
$
which is used to derive \eqref{Nane} in \cite{Nane_LIL2009}. In the proof of Theorem \ref{Theorem_main2}, we use \eqref{Laplace_E(t)} to guarantee the use of the Fubini Theorem. 
\end{em}
\end{remark}

\begin{proof}[Proof of Theorem \ref{Theorem_main2}]
By Theorem 1 of \cite{Chung} (also see the proof of Theorem 2.1 of \cite{Nane_LIL2009}), for \textit{all} $\epsilon>0$,  
\begin{align}\label{Chung}
	\mathbb{P}\Bigl(\sup_{0\le t\le 1} |W(t)|\le \epsilon \Bigr)
	=\dfrac{4}{\pi} \sum_{k=1}^\infty \dfrac{(-1)^{k-1}}{2k-1} \exp\bigg(-\dfrac{(2k-1)^2 \pi^2}{8\epsilon^2}\biggr).
\end{align}
For a fixed $\epsilon>0$, since $E$ is a continuous, nondecreasing process independent of $W$, which is self-similar with index $1/2$, a simple conditioning argument along with the use of \eqref{Chung} yields
\begin{align*}
	\mathbb{P}\Bigl(\sup_{0\le t\le T} |W(E(t))|\le \epsilon \Bigr)
	&=\mathbb{E}\biggl[\mathbb{P}\Bigl(\sup_{0\le s\le E(T)} |W(s)|\le \epsilon \Bigm| E \Bigr)\biggr]\\
	%=\mathbb{P}\Bigl(E(1)^{1/2}\sup_{0\le s\le 1} |W(s)|\le \epsilon \Bigr)\\
	&=\mathbb{E}\biggl[\mathbb{P}\Bigl(\sup_{0\le s\le 1} |W(s)|\le \dfrac{\epsilon}{E(T)^{1/2}} \Bigm| E \Bigr)\biggr]\\
%	&=\mathbb{E}\biggl[\dfrac{4}{\pi} \sum_{k=1}^\infty \dfrac{(-1)^{k-1}}{2k-1} \exp\bigg(-\dfrac{(2k-1)^2 \pi^2 E(1)}{8\epsilon^2}\biggr) \biggr]\\
	&=f_\epsilon(T),
\end{align*}
where
\begin{align*}
	f_\epsilon (t):= \dfrac{4}{\pi}\mathbb{E}\biggl[ \sum_{k=1}^\infty \dfrac{(-1)^{k-1}}{2k-1} \exp\bigg(-\dfrac{(2k-1)^2 \pi^2 E(t)}{8\epsilon^2}\biggr) \biggr].
\end{align*}
We also introduce the auxiliary function  
\begin{align*}
	g_\epsilon (t)&:= \dfrac{4}{\pi}\mathbb{E}\biggl[ \sum_{k=1}^\infty \dfrac{1}{2k-1} \exp\bigg(-\dfrac{(2k-1)^2 \pi^2 E(t)}{8\epsilon^2}\biggr) \biggr].
\end{align*}
Then by the Fubini Theorem for nonnegative integrands (applied to the product measure $\mathbb{P}\times \textrm{counting measure} \times (e^{-st}\,\mathrm{dt})$) and the formula \eqref{Laplace_E(t)}, the Laplace transform of the function $t\mapsto g_\epsilon (t)$ is given by %we observe that for $s>0$,
\begin{align*}
	\mathcal{L}_t[g_\epsilon](s)
	&=\dfrac{4}{\pi} \sum_{k=1}^\infty \dfrac{1}{2k-1} 
		\mathcal{L}_t\Biggl[\mathbb{E}\biggl[\exp\bigg(-\dfrac{(2k-1)^2 \pi^2 E(t)}{8\epsilon^2}\biggr) \biggr]\Biggr](s)\\
	&=\dfrac{4}{\pi} \sum_{k=1}^\infty \dfrac{1}{2k-1} 
		\dfrac{\psi(s)}{s}\biggl(\psi(s)+\dfrac{(2k-1)^2\pi^2}{8\epsilon^2}\biggr)^{-1}\\
	&\le \dfrac{4}{\pi} \dfrac{\psi(s)}{s}\sum_{k=1}^\infty \dfrac{8\epsilon^2}{(2k-1)^3 \pi^2}<\infty, \ \ s>0.
\end{align*}
This particularly implies that $g_\epsilon(t)<\infty$ for (Lebesgue) almost every $t>0$, but by the monotonicity of the function $g_\epsilon$, we must have $g_\epsilon(t)<\infty$ for \textit{all} $t>0$.
Therefore, due to the Fubini Theorem, the expectation and summation in the definition of $f_\epsilon(t)$ are interchangeable. Thus, 
%implies by the Fubini Theorem (for the measure $\mathbb{P}\times \textrm{counting measure}$) that
\begin{align}\label{f_e}
	\dfrac{1}{\epsilon^2}f_\epsilon(T)
	&=\dfrac{1}{\epsilon^2}\dfrac{4}{\pi} \sum_{k=1}^\infty \dfrac{(-1)^{k-1}}{2k-1} \mathbb{E}\biggl[\exp\bigg(-\dfrac{(2k-1)^2 \pi^2 E(T)}{8\epsilon^2}\biggr) \biggr]\\
	&=\dfrac{32}{\pi^3} \sum_{k=1}^\infty \dfrac{(-1)^{k-1}}{(2k-1)^3}  \varphi_T\biggl(\dfrac{(2k-1)^2\pi^2}{8\epsilon^2}\biggr),\notag
\end{align}
where $\varphi_T(a):=a\mathbb{E}[e^{-a E(T)}]$ for $a>0$. By \eqref{small_E(1)_sub} along with Lemma \ref{Lemma_Tauberian_2}, it follows that $\varphi_T(a)\to \nu(T,\infty)$ as $a\to \infty$. 
%\[
%	\sup_{k\ge 1}\varphi\biggl(\dfrac{(2k-1)^2\pi^2}{8\epsilon^2}\biggr)\le \rho+1.
%\]
Therefore, letting $\epsilon\downarrow 0$ in \eqref{f_e} and using the dominated convergence theorem (which is allowed since $\sum_{k=1}^\infty 1/(2k-1)^3<\infty$), we obtain \eqref{small_ball_2}.
%\[
%	\lim_{\epsilon\downarrow 0} \dfrac{1}{\epsilon^2}f_\epsilon(1)
%	%=\dfrac{32}{\pi^3} \sum_{k=1}^\infty \dfrac{(-1)^{k-1}}{(2k-1)^3}  \lim_{\epsilon\downarrow 0}\varphi\biggl( \dfrac{(2k-1)^2\pi^2}{8\epsilon^2}\biggr)
%	=\dfrac{32}{\pi^3} \nu(1,\infty)\sum_{k=1}^\infty \dfrac{(-1)^{k-1}}{(2k-1)^3},
%\]
\end{proof}

\begin{remark}\label{Remark_main2}
\begin{em}
In the proof of \eqref{Nane} provided in \cite{Nane_LIL2009}, where the time change is given by the inverse of a stable subordinator, the asymptotic facts about the Mittag--Leffler function  play a significant role (see equations (2.4) and (2.5) of that paper); they are employed to guarantee the use of the Fubini theorem and the dominated convergence theorem. 
For the inverse of a general \textit{non-stable} subordinator, however, the quantity $\mathbb{E}[e^{-aE(t)}]$ cannot be represented via a special function like the Mittag--Leffler function.  To overcome this difficulty, the proof provided above employs the explicit form of the Laplace transform of $t\mapsto \mathbb{E}[e^{-aE(t)}]$ (Lemma \ref{Lemma_laplace}) as well as a version of the Tauberian theorem (Lemma \ref{Lemma_Tauberian_2}) along with a general result concerning subordinators (Proposition \ref{Proposition_E}).
\end{em}
\end{remark}

Now we turn our attention to examples of time changes $E$ which are not considered in \cite{Nane_LIL2009} but can be handled by Theorem \ref{Theorem_main2}. This will entail small ball probabilities for some of the important time-changed Brownian motions representing anomalous diffusions observed in various fields of science.

Let us introduce the upper incomplete Gamma function $\Gamma(z,x)$ defined by 
\[
	\Gamma(z,x)=\int_x^\infty e^{-u}u^{z-1}\hspace{1pt} \textrm{d}u.
\]
Obviously $\Gamma(z,0)$ coincides with the Gamma function $\Gamma(z)$. Note that for $x>0$, the integral defining $\Gamma(z,x)$ is finite even when $z\le 0$. 
In particular, for $\beta\in(0,1)$ and $x>0$, a simple application of integration by parts yields
\begin{align}\label{Incomplete_Gamma}
	\Gamma(-\beta,x)=\dfrac{x^{-\beta}e^{-x}-\Gamma(1-\beta,x)}{\beta}.
\end{align}

\begin{example}[An inverse Gamma subordinator as a time change]\label{Example_Gamma}
\begin{em}
 Let $E$ be the inverse of a Gamma subordinator $D$ with parameters $c,b>0$; i.e., the Laplace exponent of $D$ in \eqref{def_LaplaceExponent} is given by 
 $
 	\psi(s)=c \log(1+s/b).
 $
  Then for all $T>0$, using the explicit form of the L\'evy measure (see e.g.\ Example 1.3.22 of \cite{Applebaum}), we obtain
  \[
  	\nu(T,\infty)= \int_T^\infty c x^{-1}e^{-bx}\hspace{1pt} \textrm{d}x =c\int_{bT}^\infty t^{-1} e^{-t}\hspace{1pt} \textrm{d}t
	=c \hspace{1pt}\Gamma(0,bT).
  \]
  Hence, \eqref{small_ball_2} with $\nu(T,\infty)$ replaced by $c \hspace{1pt}\Gamma(0,bT)$ yields the small ball probability of the time-changed Brownian motion.  
\end{em}
\end{example}

\begin{example}[An inverse tempered stable subordinator as a time change]\label{Example_tempered}
\begin{em}
Let $D$ be a tempered stable subordinator with stability index $\beta\in(0,1)$ and tempering function $q(x)$, which implies that the L\'evy measure of $D$ takes the form 
\begin{align}\label{q(x)}
	\nu(\textrm{d}x)=x^{-\beta-1}q(x)\hspace{1pt} \textrm{d}x \ \ \textrm{with} \ \ q(x)=\int_0^\infty e^{-\lambda x}\hspace{1pt} \mu(\textrm{d}\lambda),
\end{align}
where $\mu$ is a finite measure on $(0,\infty)$; see \cite{Rosinski_tempered} for details.  
By the Fubini theorem, 
\begin{align}\label{nu_tempered}
	\nu(T,\infty)
	=\int_0^\infty \int_T^\infty x^{-\beta-1} e^{-\lambda x}\hspace{1pt} \textrm{d}x\hspace{1pt} \mu(\textrm{d}\lambda)
	=\int_0^\infty \lambda^\beta \Gamma(-\beta,\lambda T)\hspace{1pt} \mu(\textrm{d}\lambda).
\end{align}
%This yields the small ball probability of the time-changed Brownian motion $W\circ E$ when $E$ is the inverse of the tempered stable subordinator $D$. 
Note that $\Gamma(-\beta,\lambda)$ has an alternative expression given by \eqref{Incomplete_Gamma}.   
\end{em}
\end{example}

\begin{remark}\label{Remark_tempered}
\begin{em}
%1) Some statistical properties of a time-changed Brownian motion $W\circ E$ with the time change $E$ being the ones discussed in Examples \ref{Example_Gamma} and \ref{Example_tempered} are found in \cite{Janczura}.
Suppose that the tempering function $q(x)$ in \eqref{q(x)} is given by the simple exponential tilting 
$	
	q(x) =\beta e^{-\lambda x}/{\Gamma(1-\beta)},
$
where $\lambda>0$ is a fixed constant. Then the Laplace exponent in \eqref{def_LaplaceExponent} takes the form $\psi(s)=(s+\lambda)^\beta-\lambda^\beta$, and using \eqref{Incomplete_Gamma}, one can write the constant $\nu(1,\infty)$ in \eqref{nu_tempered} as
\begin{align*}%\label{small_temperedstable}
	\nu(T,\infty)=\dfrac{e^{-\lambda T}T^{-\beta}-\lambda^\beta \Gamma(1-\beta,\lambda T)}{\Gamma(1-\beta)}.
\end{align*}
Letting $\lambda\downarrow 0$ yields $\nu(T,\infty)=T^{-\beta}/\Gamma(1-\beta)$, which coincides with the constant found in \eqref{nu_stable} for the inverse stable subordinator; this makes sense since a tempered stable subordinator with the tempering factor $\lambda$ set to be 0 is merely a stable subordinator.  
\end{em}
\end{remark}

\section{Extensions}\label{Section_Extensions}

This section establishes small ball probabilities for a large class of time-changed self-similar processes which includes the time-changed Brownian motions discussed in the previous section. 

Let $X=(X(t))_{t\ge 0}$ be a self-similar process starting at 0 and extend $X$ for $t<0$ using an independent copy; i.e.\ let $X'$ be an independent copy of $X$ and set $X(t):=X'(-t)$ for $t<0$. We call the so-defined process $X=(X(t))_{t\in\mathbb{R}}$ a \textit{two-sided} process. Let $E=(E(t))_{t\ge 0}$ be an independent continuous process starting at 0 which is not necessarily nondecreasing; this implies $E$ may take negative values. 
In the next theorem, the notation $f(x)\approx g(x)$ means that $0<\liminf f(x)/g(x)\le \limsup f(x)/g(x)<\infty$. 
The proof employs an idea presented in the proof of Theorem 1 of \cite{AurzadaLifshits}.

\begin{theorem}\label{Theorem_main4}
Let $X$ be a two-sided self-similar process starting at 0 of index $H>0$ such that
\begin{align}\label{small_X_weak}
	-\log \mathbb{P}\Bigl(\sup_{0\le t\le 1} |X(t)|\le \epsilon \Bigr)
	\approx \epsilon^{-\tau} \ \ \textrm{as} \ \ \epsilon\downarrow 0
\end{align}
for some $\tau>0$. 
 Let $E$ be a continuous process starting at 0, independent of $X$, such that 
\begin{align}\label{small_E(1)_weak}
	\mathbb{P}(\sup_{0\le t\le T}|E(t)|\le \epsilon) 
	\approx \epsilon^\sigma \ \ \textrm{as} \ \ \epsilon\downarrow 0
\end{align}
for some $T>0$ and $\sigma>0$. Then 
\begin{align}\label{small_ball_3_weak}
	\mathbb{P}\Bigl(\sup_{0\le t\le T} |X(E(t))|\le \epsilon \Bigr)
	\approx \epsilon^{\sigma/H} \ \ \textrm{as} \ \ \epsilon\downarrow 0.
\end{align}
\end{theorem}

%%% MULTIPLE H's --- DELETED ON 1/8/15
%%%\begin{theorem}\label{Theorem_main4}
%%%Let $H_1,\ldots,H_n$ be positive constants. 
%%%Let $X$ be a self-similar process of index $H:=\prod_{j=1}^n H_j$ such that
%%%\begin{align}\label{small_X_weak}
%%%	-\log \mathbb{P}\Bigl(\sup_{0\le t\le 1} |X(t)|\le \epsilon \Bigr)
%%%	\approx \epsilon^{-\tau} \ \ \textrm{as} \ \ \epsilon\downarrow 0,
%%%\end{align}
%%%where $\tau:=1/\sum_{i=1}^n \prod_{j=i}^n H_j$. 
%%% Let $E$ be a continuous, nondecreasing process starting at 0, independent of $X$, such that 
%%%\begin{align}\label{small_E(1)_weak}
%%%	\mathbb{P}(E(1)\le \epsilon) 
%%%	\approx \epsilon \ \ \textrm{as} \ \ \epsilon\downarrow 0.
%%%\end{align}
%%% Then 
%%%\begin{align}\label{small_ball_3_weak}
%%%	\mathbb{P}\Bigl(\sup_{0\le t\le 1} |X(E(t))|\le \epsilon \Bigr)
%%%	\approx \epsilon^{1/H} \ \ \textrm{as} \ \ \epsilon\downarrow 0.
%%%\end{align}
%%%\end{theorem}

\begin{proof}
%however, we provide a proof here for completeness of the discussion. 
For any $\theta>0$, assumption \eqref{small_E(1)_weak} is equivalent to 
\[
	\mathbb{P}(\sup_{0\le t\le T}|E(t)|^{1/\theta}\le \epsilon)%= \mathbb{P}(E(1)\le u^\theta)
	\approx \epsilon^{\theta\sigma} \ \ \textrm{as} \ \ \epsilon\downarrow 0. 
\]
which, by the weak order analogue of Lemma \ref{Lemma_Tauberian_2} (see the discussion given in Chapter V of \cite{Widder}) with $V=\sup_{0\le t\le T}|E(t)|^{1/\theta}$, implies that
\begin{align*}%\label{Tauberian2}
	\mathbb{E}[e^{-a\sup_{0\le t\le T}|E(t)|^{1/\theta}}]\approx a^{-\theta\sigma} \ \ \textrm{as} \ \ a\to\infty.
\end{align*}
This is equivalent to 
\begin{align}\label{Tauberian2}
	\mathbb{E}[e^{-a\sup_{0\le s,t\le T}|E(t)-E(s)|^{1/\theta}}]\approx a^{-\theta\sigma} \ \ \textrm{as} \ \ a\to\infty
\end{align}
due to the inequalities
\[
	\dfrac 12 \sup_{0\le s,t\le T} |E(t)-E(s)|\le \sup_{0\le t\le T}|E(t)|
	= \sup_{0\le t\le T}|E(t)-E(0)| \le \sup_{0\le s,t\le T} |E(t)-E(s)|.
\]
%where we used $E(0)=0$. 
Now, by assumption \eqref{small_X_weak}, there exist constants $c_1,c_2,\epsilon_0>0$ such that for all $\epsilon\in(0,\epsilon_0]$, 
\[
	e^{-c_1\epsilon^{-\tau}}
	\le \mathbb{P}\Bigl(\sup_{0\le t\le 1} |X(t)|\le \epsilon \Bigr)
	\le e^{-c_2\epsilon^{-\tau}}.
\]
Setting $c_3:=e^{-c_1 \epsilon_0^{-\tau}}$ and $c_4:=e^{c_2 \epsilon_0^{-\tau}}$,
 we observe that for \textit{all} $\epsilon>0$, 
\begin{align}\label{key_estimate}
	c_3e^{-c_1\epsilon^{-\tau}}
	\le \mathbb{P}\Bigl(\sup_{0\le t\le 1} |X(t)|\le \epsilon \Bigr)
	\le c_4e^{-c_2\epsilon^{-\tau}}.
\end{align}
Let $N:=\inf_{0\le t\le T} E(t)$ and $M:=\sup_{0\le t\le T} E(t)$. The assumption that $E(0)=0$ implies that $N\le 0$ and $M\ge 0$. For $\epsilon>0$, using continuity of $E$, independence between $(X(t))_{t> 0}$ and $(X(t))_{t<0}$, independence between $X$ and $E$, and the self-similarity of $X$, we observe that
\begin{align*}
	\mathbb{P}\Bigl(\sup_{0\le t\le T} |X(E(t))|\le \epsilon \Bigr)
	&=\mathbb{P}\Bigl(\sup_{N\le s\le 0} |X(s)|\le \epsilon, \sup_{0\le s\le M} |X(s)|\le \epsilon \Bigr)\\
	&=\mathbb{E}\biggl[\mathbb{P}\Bigl(\sup_{N\le s\le 0} |X(s)|\le \epsilon \Big| E \Bigr) \mathbb{P}\Bigl( \sup_{0\le s\le M} |X(s)|\le \epsilon \Big| E \Bigr)\biggr]\\
	%&=\epsilon^{-1/H}\mathbb{E}\biggl[\mathbb{P}\Bigl(\sup_{0\le s\le E(1)} |X(s)|\le \epsilon \hspace{1pt}\Big|\hspace{1pt} E \Bigr)\biggr]\\
	&=\mathbb{E}\biggl[\mathbb{P}\Bigl(\sup_{0\le s\le 1} |X(s)|\le \dfrac{\epsilon}{(-N)^H}  \hspace{1pt}\Big|\hspace{1pt} E \Bigr)\mathbb{P}\Bigl(\sup_{0\le s\le 1} |X(s)|\le \dfrac{\epsilon}{M^H}  \hspace{1pt}\Big|\hspace{1pt} E \Bigr)\biggr].
\end{align*}
By the upper bound in \eqref{key_estimate} and the elementary inequality 
%\[
%	(x+y)^{\tau H}\le (2^{\tau H-1}\vee 1) \{x^{\tau H}+y^{\tau H}\}, \ x,y\ge 0,
%\]
$(x+y)^{p}\le d(p) (x^{p}+y^{p})$, $x,y\ge 0$, with $p=\tau H$, it follows that
\begin{align*}
	\epsilon^{-\sigma/H}\mathbb{P}\Bigl(\sup_{0\le t\le T} |X(E(t))|\le \epsilon \Bigr)
	&\le c_4^2\epsilon^{-\sigma/H}\mathbb{E}\Bigl[e^{-c_2\epsilon^{-\tau}\{(-N)^{\tau H}+M^{\tau H}\}} \Bigr]\\	
	&\le c_4^2\epsilon^{-\sigma/H}\mathbb{E}\Bigl[e^{-\bar{c}_2 \epsilon^{-\tau}(M-N)^{\tau H}} \Bigr]\\
	&= c_4^2\epsilon^{-\sigma/H}\mathbb{E}\Bigl[e^{-\bar{c}_2 \epsilon^{-\tau}\sup_{0\le s,t\le T}|E(t)-E(s)|^{\tau H}} \Bigr]\\
	&=\dfrac{c_4^2}{(\bar{c}_2)^{\theta\sigma}}\varphi_{T,\theta,\sigma}(\bar{c}_2\epsilon^{-\tau}),
\end{align*}
where $\bar{c}_2:=c_2/d(\tau H)$, $\theta:=1/(\tau H)$ and 
\[ 
	\varphi_{T,\theta,\sigma}(a):=a^{\theta\sigma} \mathbb{E}[e^{-a \sup_{0\le s,t\le T}|E(t)-E(s)|^{1/\theta}}].
\]
Now, \eqref{Tauberian2} implies that $\limsup_{a\to\infty} \varphi_{T,\theta,\sigma}(a)<\infty$, and hence, the desired upper bound follows.  
%\[
%	\limsup_{\epsilon\downarrow 0}\epsilon^{-\sigma/H}\mathbb{P}\Bigl(\sup_{0\le t\le 1} |X(E(t))|\le \epsilon \Bigr)<\infty. 
%\]
The lower bound is obtained in a similar manner. 
\end{proof}

\begin{remark}\label{Remark_main4}
\begin{em}
1) The rate of decay of the small ball probability of $X\circ E$ in \eqref{small_ball_3_weak} does not depend on $\tau$ appearing in \eqref{small_X_weak}; the information of $\tau$ is reflected on the constant $\bar{c}_2$ introduced in the proof. 

2) Unlike Theorem 4 of \cite{AurzadaLifshits}, a simple modification of the above proof does not lead to a similar result concerning strong deviation orders (i.e.\ a result with $\approx$ replaced by $\sim$). Indeed, if we assume (instead of \eqref{small_X_weak}) that
\begin{align*}
	-\log \mathbb{P}\Bigl(\sup_{0\le t\le 1} |X(t)|\le \epsilon \Bigr)
	\sim k \epsilon^{-\tau} \ \ \textrm{as} \ \ \epsilon\downarrow 0
\end{align*}
for some $k>0$, then for any $\delta\in(0,1)$, we can find constants $c_3$ and $c_4$ such that 
\[
	c_3 e^{-k  (1+\delta)\epsilon^{-\tau}}
	\le \mathbb{P}\Bigl(\sup_{0\le t\le 1} |X(t)|\le \epsilon \Bigr)
	\le c_4 e^{-k  (1-\delta)\epsilon^{-\tau}}
\]
for all $\epsilon>0$. This leads to 
\begin{align}\label{observation_strong}
	\dfrac{c_3^2}{(\bar{c}_1)^{\theta\sigma}}\varphi_{T,\theta,\sigma}(\bar{c}_1\epsilon^{-\tau})
	\le \epsilon^{-\sigma/H}\mathbb{P}\Bigl(\sup_{0\le t\le T} |X(E(t))|\le \epsilon \Bigr)
	\le \dfrac{c_4^2}{(\bar{c}_2)^{\theta\sigma}}\varphi_{T,\theta,\sigma}(\bar{c}_2\epsilon^{-\tau}),
\end{align}
where  $\bar{c}_1:=k(1+\delta)/d(\tau H)$, $\bar{c}_2:=k(1-\delta)/d(\tau H)$, and $\varphi_{T,\theta,\sigma}(a)$ is as in the above proof. Now, if we further assume a strong deviation condition for the time change $E$, then $\varphi_{T,\theta,\sigma}(a)$ approaches a constant as $a\to \infty$; however, since the constants $c_3$ and $c_4$ depend on $\delta$ and do not generally approach the same value as $\delta\to 0$,  a strong  result for the small ball probability for $X\circ E$ does not follow from \eqref{observation_strong}. 
Note that this issue does not occur in the proof of Theorem 4 of \cite{AurzadaLifshits} since the logarithmic deviation is discussed in that theorem. 
%for $\delta\in(0,1)$, there exists a constant $\epsilon_0(\delta)>0$ such that for all $\epsilon\in(0,\epsilon_0(\delta))$, 
%\[
%	e^{-k  (1+\delta)\epsilon^{-\tau}}
%	\le \mathbb{P}\Bigl(\sup_{0\le t\le 1} |X(t)|\le \epsilon \Bigr)
%	\le e^{-k  (1-\delta)\epsilon^{-\tau}}.
%\]
On the other hand, 
%as for the time-changed Brownian motion 
in Theorem \ref{Theorem_main2}, the explicit formula for the small ball probability of the Brownian motion (valid for each fixed $\epsilon>0$) allowed us to establish a strong deviation result. 
\end{em}
\end{remark}

We now consider some specific outer processes $X$ that can be handled within the setting of Theorem \ref{Theorem_main4}. Some of the examples below show that Theorem \ref{Theorem_main4} indeed generalizes Theorem 2.3 of \cite{Nane_LIL2009}.

Well-known examples of self-similar processes which have logarithmic small deviation orders include a fractional Brownian motion and a symmetric stable L\'evy process. Namely, if $W_H$ denotes a fractional Brownian motion with Hurst index $H\in(0,1)$, i.e.\ $W_H$ is a zero mean Gaussian process with covariance function $\mathbb{E}[W_H(s)W_H(t)]= (s^{2H}+t^{2H}-|s-t|^{2H})/2$, then % In the case when $H=1/2$, the process $W_H$ becomes a Brownian motion. 
$W_H$ is a self-similar process of index $H$ with small deviation order given by
\begin{align*}%\label{small_fBm1}
	-\log \mathbb{P}\Bigl(\sup_{0\le t\le 1} |W_H(t)|\le \epsilon \Bigr)
	\sim c_H \hspace{1pt}\epsilon^{-1/H} \ \ \textrm{as} \ \ \epsilon\downarrow 0,
\end{align*}
where $c_H$ is a positive constant depending on $H$. An explicit representation of the small deviation constant $c_H$ is found in 
%Li and Linde 
\cite{LiLinde1998}. 

On the other hand, if $S_\alpha$ is a symmetric stable L\'evy process of stability index $\alpha\in(0,2]$, i.e.\ $S_\alpha$ is a L\'evy process with characteristic function 
$
	\mathbb{E}[e^{iu S_\alpha(t)}]=e^{- t \kappa^\alpha |u|^\alpha}
$
for some positive constant $\kappa$ (see e.g.\ \cite{Applebaum,Sato}), then $S_\alpha$ is a self-similar process of index $H=1/\alpha$ and
\[	
	-\log \mathbb{P}\Bigl(\sup_{0\le t\le 1} |S_\alpha(t)|\le \epsilon \Bigr)
	\sim \lambda_\alpha\hspace{1pt} \epsilon^{-\alpha} \ \ \textrm{as} \ \ \epsilon\downarrow 0,
\]
where $\lambda_\alpha>0$ is some constant; see \cite{Mogulskii} for details. 
%the principle Dirichlet eigenvalue for the fractional Laplacian operator associated with $S_\alpha$ in the interval $[-1,1]$; see Mogul'skii \cite{Mogulskii} for details. 

Both of these examples satisfy condition \eqref{small_X_weak} with $\tau=1/H$, and they can also be handled by Theorem 2.3 of \cite{Nane_LIL2009}. However, self-similar processes with index $H$ with $\tau\ne 1/H$ also exist as the following examples show. These processes are outside the scope of Theorem 2.3 of \cite{Nane_LIL2009}, but Theorem \ref{Theorem_main4} still applies.

%Theorem \ref{Theorem_main3} entails some immediate corollaries. 
%%
%%\begin{example}[A fractional Brownian motion as an outer process]\label{Example_fBm}
%%\begin{em}
%%Let $W_H$ be a fractional Brownian motion with Hurst index $H\in(0,1)$; i.e.\ $W_H$ is a zero mean Gaussian process with covariance function $\mathbb{E}[W_H(s)W_H(t)]= (s^{2H}+t^{2H}-|s-t|^{2H})/2$. In the case when $H=1/2$, the process $W_H$ becomes a Brownian motion. $W_H$ is a self-similar process with index $H$ with small deviation order given by
%%\begin{align*}%\label{small_fBm1}
%%	-\log \mathbb{P}\Bigl(\sup_{0\le t\le 1} |W_H(t)|\le \epsilon \Bigr)
%%	\sim c_H \hspace{1pt}\epsilon^{-1/H} \ \ \textrm{as} \ \ \epsilon\downarrow 0,
%%\end{align*}
%%where $c_H$ is a positive constant depending on $H$; see Li and Linde \cite{LiLinde1998} for an explicit representation of the small deviation constant $c_H$. Clearly, $X:=W_H$ satisfies the assumption of Theorem \ref{Theorem_main4} with $\tau=1/H$. 
%%\hfill $\blacksquare$
%%\end{em}
%%\end{example}

\begin{example}[An iterated fractional Brownian motion as an outer process]
\begin{em}
%More generally, one can consider an iterated (two-sided) fractional Brownian motion, which is defined as follows. 
%Let $W_H=(W_H(t))_{t\ge 0}$ be a fractional Brownian motion with Hurst index $H\in(0,1)$. Extend $W_H$ for $t<0$ using an independent copy $\widetilde{W}_H$; i.e.\ for $t<0$, let $W_H(t):=\widetilde{W}_H(-t)$. We call the so-defined process $W_H=(W_H(t))_{t\in\mathbb{R}}$ a two-sided fractional Brownian motion. 
An \textit{$n$-iterated two-sided fractional Brownian motion} is the process $X^{(n)}$ defined by the iteration
\[
	X^{(1)}(t):=W_{H_1}(t), \ \  X^{(j)}(t):=W_{H_{j}}(X^{(j-1)}(t)), \ \ j=2,\ldots, n,
\]
where $W_{H_1},\ldots,W_{H_n}$ are independent two-sided fractional Brownian motions with Hurst indices $H_1,\ldots,H_n$ and small deviation constants $c_{H_1},\ldots,c_{H_n}$, respectively. 
The process $X^{(n)}$ is self-similar with index $H_{(n)}:=\prod_{j=1}^n H_j$. Moreover, it is established in Section 4.2 of \cite{AurzadaLifshits} that 
\begin{align*}%\label{small_fBm2}
	-\log \mathbb{P}\Bigl(\sup_{0\le t\le 1} |X^{(n)}(t)|\le \epsilon \Bigr)
	\sim c_n \hspace{1pt}\epsilon^{-\tau_n} \ \ \textrm{as} \ \ \epsilon\downarrow 0,
\end{align*}
where $\tau_n:=1/\sum_{i=1}^n \prod_{j=i}^n H_j$ and $c_n$ is defined iteratively by
\[
	c_1:= c_{H_1}, \ \ c_j:=(1+\tau_{j-1})\Bigl[c_{j-1}^{1/\tau_{j-1}}\dfrac{2 c_{H_j}}{\tau_{j-1}}\Bigr]^{\tau_{j-1}/(1+\tau_{j-1})}, \ \ j=2,\ldots, n.
\]
Hence, condition \eqref{small_X_weak} holds with $\tau=\tau_n\ne 1/H_{(n)}$. 
%hence, Theorems \ref{Theorem_main3} and \ref{Theorem_main4} indeed generalize the class of outer processes. 
 %setting $n=1$ reduces \eqref{small_fBm2} to \eqref{small_fBm1}.  Moreover, 
\end{em}
\end{example}

%%
%%\begin{example}[A symmetric stable L\'evy process as an outer process]
%%\begin{em}
%%Let $S_\alpha$ be a symmetric stable L\'evy process of stability index $\alpha\in(0,2]$. This means that $S_\alpha$ is a L\'evy process with characteristic function 
%%$
%%	\mathbb{E}[e^{iu S_\alpha(t)}]=e^{- t \kappa^\alpha |u|^\alpha}
%%$
%%for some positive constant $\kappa$ (see e.g.\ \cite{Applebaum,Sato}). It is known that $S_\alpha$ is a self-similar process of index $H=1/\alpha$ and
%%\[	
%%	-\log \mathbb{P}\Bigl(\sup_{0\le t\le 1} |S_\alpha(t)|\le \epsilon \Bigr)
%%	\sim \lambda_\alpha\hspace{1pt} \epsilon^{-\alpha} \ \ \textrm{as} \ \ \epsilon\downarrow 0,
%%\]
%%where $\lambda_\alpha>0$ is the principle Dirichlet eigenvalue for the fractional Laplacian operator associated with $S_\alpha$ in the interval $[-1,1]$; see Mogul'skii \cite{Mogulskii} for details. 
%%\hfill $\blacksquare$ 
%%\end{em}
%%\end{example}
%%

\begin{example}[An iterated strictly stable L\'evy process as an outer process]
\begin{em}
Let $S_{\alpha_1}$ be a two-sided strictly stable L\'evy process of index $\alpha_1\in(0,2]$. Let $S_{\alpha_2}$ be an independent strictly stable L\'evy process of index $\alpha_2\in(0,2]$ which is not a subordinator. We call the process $X:=S_{\alpha_1}\circ S_{\alpha_2}$ an \textit{iterated strictly stable L\'evy process}. It is easy to see that $X$ is self-similar with index $H=1/(\alpha_1\alpha_2)$. Moreover, it is shown in Section 5 of \cite{AurzadaLifshits} that
\[	
	-\log \mathbb{P}\Bigl(\sup_{0\le t\le 1} |S_{\alpha_1}(S_{\alpha_2}(t))|\le \epsilon \Bigr)
	\approx \epsilon^{-\alpha_1 \alpha_2/(1+\alpha_2)} \ \ \textrm{as} \ \ \epsilon\downarrow 0.
\]
Hence, condition \eqref{small_X_weak} holds with $\tau=\alpha_1 \alpha_2/(1+\alpha_2)\ne \alpha_1\alpha_2=1/H$.  
%the iterated strictly stable L\'evy process satisfies the assumption of Theorem \ref{Theorem_main4}. %%%% with $n=2$, $H_1=1/\alpha_2$ and $H_2=1/\alpha_1$. 
%Note that this process cannot be handled by Theorem 2.3 of \cite{Nane_LIL2009}.  
\end{em}
\end{example}

Theorem \ref{Theorem_main4} also allows us to consider time changes which are given by mixtures of independent inverse subordinators.

\begin{example}[A mixture of independent inverse subordinators as a time change]
\begin{em}
For each $j=1,\ldots,m$, let $E_j$ be the inverse of a subordinator $D_j$ with infinite L\'evy measure $\nu_j$ having no atom at $T>0$ so that  \eqref{small_E(1)_sub} holds for each $E_j$. Assume that $E_j$'s are independent and let $E:=\sum_{j=1}^m c_jE_j$, where $c_j$'s are positive constants. Then it follows from Lemma \ref{Lemma_Tauberian_2} that
\[
	\mathbb{E}[e^{-a E(T)}]=\prod_{j=1}^m \mathbb{E}[e^{-ac_j E_j(T)}]\sim \biggl(\prod_{j=1}^m \dfrac{\nu_j(T,\infty)}{c_j}\biggr) a^{-m}  \ \ \textrm{as} \ \ a\to \infty,
\]
which, again by Lemma \ref{Lemma_Tauberian_2}, is equivalent to 
\[
	\mathbb{P}(E(T)\le \epsilon)\sim \dfrac{1}{m!}\biggl(\prod_{j=1}^m \dfrac{\nu_j(T,\infty)}{c_j}\biggr) \epsilon^m \ \ \textrm{as} \ \ \epsilon\downarrow 0.
\]
Hence, Theorem \ref{Theorem_main4} applies to the time change $E$ with $\sigma=m$.  Moreover, with this specific time change, it is possible to generalize Theorems \ref{Theorem_main2}  to obtain the small ball probability of the time-changed Brownian motion with the exact small deviation constant specified. The proof simply combines the ideas used in the proofs of Theorems \ref{Theorem_main2} and \ref{Theorem_main4} and hence is omitted.

Note that even if each $D_j$ is a stable subordinator, the time change $E$ defined in this example does \textit{not} coincide with the inverse of a mixture of independent stable subordinators appearing in \cite{HKU-2,HKU-1,MS_2}. Indeed, in those papers, $E$ is defined to be the inverse of $D:=\sum_{j=1}^m c_jD_j$, where $D_j$'s are independent stable subordinators, which implies that it has the small ball probability with $\sigma=1$ due to Proposition \ref{Proposition_E}. 
\end{em}
\end{example}

\section*{Appendix}\label{Appendix}

\begin{proof}[Proof of Proposition \ref{Proposition_E}]
%The assumption that the L\'evy measure is infinite implies that $D$ has strictly increasing sample paths (see e.g.\ Theorem 21.3 of \cite{Sato} or Proposition 1.3 of \cite{Bertoin_subordinator}), which in turn implies that the inverse $E$ is a continuous, nondecreasing process.
By the inverse relationship between $E$ and $D$,
% along with the assumption for the L\'evy measure $\nu$, 
it follows that $\mathbb{P}(E(T)\le \epsilon)=\mathbb{P}(D(\epsilon)\ge T)$. 
Hence, we only need to verify that 
\begin{align*}
	\lim_{\epsilon\downarrow 0}\dfrac{\mathbb{P}(D(\epsilon)\ge T)}{\epsilon}
	=\nu(T,\infty) \ \ \textrm{provided that} \ \ \nu(\{T\})=0.
\end{align*}
Although this may be a well-known fact, for the sake of completeness of the discussion as well as clarification of why the assumption that $\nu(\{T\})=0$ is needed, we provide a proof below. Note that a similar argument appears in \cite{Hettmansperger}.

For a fixed real sequence $\{\epsilon_n\}$ with $\epsilon_n\downarrow 0$, let 
\[
	\bar{\nu}_n(x):=\dfrac{\mathbb{P}(D(\epsilon_n)\ge x)}{\epsilon_n} \ \ \textrm{and} \ \ 
	\bar{\nu}(x):=\nu(x,\infty)
\]
 for $x>0$.
Then the proof of Theorem 1.2(i) of \cite{Bertoin_subordinator} shows that the sequence of absolutely continuous measures $\bar{\nu}_n(x)\hspace{1pt}\textrm{d}x$ converges vaguely to $b\hspace{1pt} \delta_0(\textrm{d}x)+\bar{\nu}(x)\hspace{1pt}\textrm{d}x$, where $b\ge 0$ is the drift parameter appearing in \eqref{def_LaplaceExponent} and $\delta_0$ is the Dirac measure with mass at 0. This particularly implies that for any $0<c<d$, 
\[
	\lim_{n\to \infty}\int_c^d \bar{\nu}_n (x)\hspace{1pt}\textrm{d}x=\int_c^d \bar{\nu} (x)\hspace{1pt}\textrm{d}x.
\]

Now, assume that $\bar{\nu}_n(T)$ does not converge to $\bar{\nu}(T)$. Then there exist a constant $\eta>0$ and a subsequence $\{n_k\}$ such that $|\bar{\nu}_{n_k}(T)-\bar{\nu}(T)|\ge \eta$ for all $k$. This implies that there exists a further subsequence $\{n_{k_m}\}$ such that (i) $\bar{\nu}_{n_{k_m}}(T)\ge \bar{\nu}(T)+\eta$ for all $m$ or (ii) $\bar{\nu}_{n_{k_m}}(T)\le \bar{\nu}(T)-\eta$ for all $m$. If (i) holds, then since each $\bar{\nu}_{n_{k_m}}(\cdot)$ is a decreasing function, for any $\delta\in(0,T)$,
\begin{align}\label{nubar_1}
	\int_{T-\delta}^T \bar{\nu} (x)\hspace{1pt} \textrm{d}x
	=\lim_{m\to \infty}\int_{T-\delta}^T \bar{\nu}_{n_{k_m}} (x)\hspace{1pt} \textrm{d}x
	\ge \lim_{m\to \infty}\delta \bar{\nu}_{n_{k_m}}(T)
	\ge \delta(\bar{\nu}(T)+\eta). 
\end{align}
On the other hand, the assumption that $\nu(\{T\})=0$ implies that $\bar{\nu}(\cdot)$ is continuous at $x=T$; hence, there exists a constant $\delta\in(0,T)$ such that $\bar{\nu}(x)-\bar{\nu}(T)\le \eta/2$ for all $x$ with $T-\delta\le x\le T$. Thus, we have
\[
	\int_{T-\delta}^T \bar{\nu} (x)\hspace{1pt} \textrm{d}x\le \delta(\bar{\nu}(T)+\eta/2),
\]
which contradicts the estimate in \eqref{nubar_1}. A similar contradiction occurs if (ii) holds. Therefore, $\bar{\nu}_n(T)$ must converge to $\bar{\nu}(T)$, which completes the proof. 
\end{proof}

\begin{proof}[Proof of Lemma \ref{Lemma_laplace}]
For a fixed $x>0$, since $E$ is the inverse of $D$, 
\[
  \mathbb{P}(E(t)\le x)
  =\mathbb{P}(D(x)\ge t)
   =1-\mathbb{P}(D(x)< t), \ \ t>0.
\]
Taking the Laplace transform with respect to $t$ on both sides, we obtain 
\[
	\mathcal{L}_t\bigl[\mathbb{P}(E(t)\le x)\bigr](s)
	=\dfrac{1}{s}-\dfrac{1}{s}\mathcal{L}_t[\mathbb{P}(D(x)\in \textrm{d}t)](s) 
	=\dfrac{1-\mathbb{E}[e^{-s D(x)}]}{s}
	=\dfrac{1-e^{-x\psi(s)}}{s}, \ \ s>0,
\]
where $\mathcal{L}_t[f(t)]$ and $\mathcal{L}_t[\mu(\textrm{d}t)]$ denote the Laplace transforms of a function $f(t)$ and a measure $\mu(\textrm{d}t)$, respectively. 
The right hand side of the above identity being differentiable with respect to $x$, so is the left hand side, and 
\begin{align*}
	\mathcal{L}_t\bigl[\mathbb{P}(E(t)\in \textrm{d}x)\bigr](s)
	=\dfrac{\psi(s)}{s}e^{-x\psi(s)}\,\textrm{d}x, \ \ s>0.
\end{align*}
Hence, we obtain by the Fubini theorem (for nonnegative integrands) that 
\begin{align*}
	\mathcal{L}_t \big[\mathbb{E}[e^{-a E(t)}]\bigr](s)
	= \int_0^\infty \dfrac{\psi(s)}{s}e^{-x(\psi(s)+a)} \,\textrm{d}x
	=\dfrac{\psi(s)}{s}(\psi(s)+a)^{-1}, \ \ s>0,
\end{align*}
which completes the proof.
\end{proof}

\vspace{3mm}
\noindent
{\textbf{Acknowledgments:}}
The author would like to thank Professor Jan Rosi\'nski and Professor Xia Chen of the University of Tennessee  for helpful discussions.

  \bibliographystyle{plain} 
  \bibliography{KobayashiK}

\end{document}